\documentclass[pdflatex,sn-mathphys-num]{sn-jnl}

\usepackage{graphicx}%
\usepackage{multirow}%
\usepackage{amsmath,amssymb,amsfonts}%
\usepackage{amsthm}%
\usepackage{mathrsfs}%
\usepackage[title]{appendix}%
\usepackage{xcolor}%
\usepackage{textcomp}%
\usepackage{manyfoot}%
\usepackage{booktabs}%
\usepackage{algorithm}%
\usepackage{algorithmicx}%
\usepackage{algpseudocode}%
\usepackage{listings}%

\theoremstyle{thmstyleone}%
\newtheorem{theorem}{Theorem}%
\newtheorem{proposition}[theorem]{Proposition}%

\newtheorem{corollary}[theorem]{Corollary}%
\newtheorem{remark}{Remark}%
\newtheorem{lemma}[theorem]{Lemma}%
\newtheorem{conjecture}{Conjecture}%

\theoremstyle{thmstylethree}%
\def\C{\mathbb{C}}
\def\N{\mathbb{N}}
\def\R{\mathbb{R}}

\begin{document}

\title[Exceptional Component]{The Geometry of the Exceptional Component of Degree Two Foliations on $\mathbb{P}^3$}

\author*[1]{\fnm{Claudia R.} \sur{Alcántara}}

\author[2]{\fnm{Dominique} \sur{Cerveau}}



\affil*[1]{\orgdiv{Departamento de Matemáticas},
\orgname{Universidad de Guanajuato},
\orgaddress{\city{Guanajuato}, \postcode{36024},
\state{Guanajuato}, \country{Mexico.}}
E-mail: claudia@cimat.mx}

\affil[2]{\orgdiv{Université de Rennes I},
\orgname{IRMAR},
\orgaddress{\street{Campus Beaulieu}, \city{Rennes Cedex},
\postcode{35042}, \country{France.}}
E-mail: dominique.cerveau@univ-rennes1.fr}

\abstract{We study the exceptional component of the space $\mathbb{F}(2,\mathbb{P}^3)$, of codimension-one foliations of degree two on $\mathbb{P}^3$. We describe the geometry of its boundary and prove that it has four irreducible components, all of dimension $12$. Three of these components contain a dense subset given by the orbit of a logarithmic foliation of type $(1,1,2)$, while the fourth contains a family of pull-back type foliations from $\mathbb{P}^2$ whose orbits have dimension $11$.}

\keywords{Holomorphic foliations; exceptional component; representation weights; degenerations.}

\pacs[MSC Classification]{37F75 (primary); 14L30, 20G05}

\maketitle

\section{Introduction}

The study of the geometry of the irreducible components of the space of codimension-one foliations on $\mathbb{P}^3$ is an important problem in the theory of holomorphic foliations. In particular, understanding the structure of these components and their boundaries is a natural question.

Let $\mathbb{F}(2,\mathbb{P}^3)$ denote the space of codimension-one foliations of degree two on $\mathbb{P}^3$. In this paper, we focus on one of the irreducible components of $\mathbb{F}(2,\mathbb{P}^3)$. The approach developed here provides methods that may be applicable to the study of the geometry and boundaries of the remaining components.
\\

In \cite{cl}, Cerveau and Lins Neto proved that $\mathbb{F}(2,\mathbb{P}^3)$ has six irreducible components. These components are Zariski-open subsets of Zariski-closed subsets of a projective space, and although the generic foliation in each of them is known, their boundaries are still not completely understood. 

Among these components is the exceptional component, defined as the closure of the $SL_4(\mathbb{C})$-orbit of a distinguished foliation $\mathcal{F}_{\omega}$. Despite this seemingly simple description, the geometry of this component remains only partially understood. In particular, it has long been unknown which foliations belong to its boundary. The main goal of this work is to describe the geometry of the boundary of the exceptional component.

Some properties of the exceptional component are already known. For instance, it is a $13$-dimensional algebraic variety, and its degree is $168208$, a result established by Vainsencher and Rossini \cite{RossiniVainsencher}. In this work, we exploit the fact that the exceptional component contains a dense $SL_4(\mathbb{C})$-orbit, and we apply techniques from algebraic group theory and geometric invariant theory to study its geometry. Our first main result shows that the boundary of the exceptional component is the union of four irreducible components of dimension $12$.

A second key ingredient is the representation of $SL_4(\mathbb{C})$ on the vector space of twisted differential forms. Using this representation, we obtain an explicit description of the foliations lying in the boundary of the exceptional component. Moreover, the irreducible component can be visualized as a polytope contained in the weight polytope of the full representation space of twisted differential forms. From this perspective, the boundary corresponds to weight vectors associated with the faces of this polytope, which generate the degenerations appearing in the boundary of the component.
We explicitly describe the one-parameter subgroups whose degenerations yield all foliations in the boundary.
\\

To obtain the results, we will proceed as follows. We first prove that every foliation in the boundary of the exceptional component is, up to a change of coordinates, contained in the closure of the orbit of the group $P$, of upper triangular matrices. This reduces the study of the boundary of $\overline{SL_4(\mathbb{C})\cdot \omega}$ to the analysis of $\overline{P\cdot \omega}$, where all boundary foliations can already be detected. We then show that the closure of this $P$-orbit has four irreducible components. For three of these components, we obtain an open subset that corresponds to the orbit of a foliation. For the remaining component, we describe a family of foliations arising as degenerations of the rational first integral of the exceptional foliation; these foliations are of pull-back type.
\\

The paper is organized as follows. In Section~2 we present the necessary preliminaries. We define the space of codimension-one foliations on $\mathbb{P}^3$ and describe the irreducible components of the space of degree-two codimension-one foliations on $\mathbb{P}^3$. We also study the $\mathfrak{sl}_4$-representation associated with the natural action of changes of coordinates on the space of foliations.

In Section~3 we give an explicit description of the automorphism group of the exceptional foliation, whose orbit closure defines the exceptional component. We also obtain a decomposition of this group, this decomposition will play a fundamental role in the subsequent constructions.

In Section 4, we prove that, up to a change of coordinates, all foliations in the boundary can be obtained by considering only the action of the subgroup of upper triangular matrices in $SL_4(\C)$. Using the geometry of the orbit under this group, we show that the boundary of the exceptional component contains four irreducible components.

Finally, Section~5 is devoted to determining all the foliations contained in the irreducible components of the boundary of the exceptional component. To this end, we use the representation to construct the associated weight polytope and determine all of its faces. Each face corresponds to a one-parameter subgroup, which we explicitly determine, together with the corresponding limiting foliations.

\section{Preliminaires}

Let $s \in \N \cup \{0\}$ and consider the vector space:
\begin{align*}
&H^0\big(\mathbb{P}^3, \Omega_{\mathbb{P}^3}^1(s+2)\big)=\Big\{\omega=A_1dz_1+A_2dz_2+A_3dz_3+A_4dz_4: \\&\textrm{$A_i \in \C[z_1,z_2,z_3,z_4]$ is homogeneous of degree $s+1$ for $i=1,2,3,4$ and $\sum_{i=1}^4 z_iA_i=0$}\Big\}
\end{align*}

\noindent A holomorphic foliation $\mathcal{F}$ of codimension-one and degree $s$ on $\mathbb{P}^3$  is given by the projective class of a 1-form, $[\omega] \in \mathbb{P}(H^0(\mathbb{P}^3, \Omega_{\mathbb{P}^3}^1(s+2))$, such that it satisfy the integrability condition:  $\omega \wedge d\omega=0$. If $\mathcal{F}_{\omega}$ is the foliation given by $\omega=\sum_{i=1}^4 A_idz_i$ then its singular set is the algebraic variety $\mathbb{V}(A_1,A_2,A_3,A_4) \subset \mathbb{P}^3$. We denote by $\mathbb{F}(s,\mathbb{P}^3)$ the space of foliations with singular set of codimension $\geq 2$.
\\

Then $\mathbb{F}(s,\mathbb{P}^3)$ can be identified with a Zariski's locally closed set in the projective space $\mathbb{P}(H^0(\mathbb{P}^3, \Omega_{\mathbb{P}^3}^1(s+2))$, which has dimension $4 \binom{s+4}{3} - \binom{s+5}{3}-1$. 
\\

It is known from the work by Cerveau and Lins Neto \cite{cl} that $\mathbb{F}(2,\mathbb{P}^3)$ has six irreducible components. We now describe the generic element of each of these components.

\begin{enumerate}
\item $S(2)$: linear pull-back of a foliation on $\mathbb{P}^2$ of degree $2$.
\item $\overline{R(2,2)}$: foliation with a rational first integral $\frac{f}{g}$, where $f, g \in \C[z_1,z_2,z_3,z_4]$ have degree 2 and are irreducible.
\item $\overline{R(1,3)}$: foliation with a rational first integral $\frac{f}{L^3}$, where $f$ has degree $3$ and $L$ has degree $1$.
\item $\overline{L(1,1,1,1)}$: logarithmic foliation given by the 1-forms:

$$\omega=L_1L_2L_3L_4 \sum_{i=1}^4 \lambda_i  \frac{dL_i}{L_i},$$

\noindent where $L_1, L_2, L_3$ and $L_4$ in $\C[z_1,z_2,z_3,z_4]$ have degree $1$ and $\sum_{i=1}^4 \lambda_i=0$.

\item $\overline{L(1,1,2)}$: logarithmic foliations given by the 1-forms:

$$\omega=f_1f_2f_3 \sum_{i=1}^3 \lambda_i  \frac{df_i}{f_i},$$

\noindent where $f_1, f_2$ define different hyperplanes, $f_3$ defines an irreducible hypersurface of degree $2$ and $\lambda_1+\lambda_2+2\lambda_3=0$.

\item The exceptional component $\overline{E}$.
\end{enumerate}

The exceptional component is the unique irreducible component that is not obtained by the standard logarithmic or pull-back constructions. It is defined as the Zariski closure of the orbit of a single foliation $\omega$ under the natural action of $SL_4(\mathbb{C})$, that is,
\[
\overline{E} \;=\; \overline{SL_4(\mathbb{C}) \cdot \omega}.
\]

\subsection{The Representation $H^0(\mathbb{P}^3, \Omega_{\mathbb{P}^3}^1(4))$ of $SL_4(\C)$}

In this subsection we will study the representation of the Lie algebra $\mathfrak{sl}_4$ of $SL_4(\C)$ given by the linear action:

\begin{align*}
SL_4(\C) \times H^0(\mathbb{CP}^3, \Omega_{\mathbb{CP}^3}^1(4)) &\to H^0(\mathbb{CP}^3, \Omega_{\mathbb{CP}^3}^1(4))\\
\Big((g_{ij})^{-1},  \sum_{i=1}^4 A_idz_i\Big) &\mapsto \sum_{i=1}^4 A_i((g_{ij})(z_1,z_2,z_3,z_4))(g_{i1}dz_1+g_{i2}dz_2+g_{i3}dz_3+g_{i4}dz_4).
\end{align*}


This representation plays an important role in the description of geometric properties not only of the irreducible component studied in this work, but also of all irreducible components of the space of foliations on $\mathbb{F}(2,\mathbb{P}^3)$.

Let $V=\mathbb{P}(H^0(\mathbb{P}^3, \Omega_{\mathbb{P}^3}^1(4))$, then the above action induces a representation:

\begin{align*}
\mathfrak{sl}_4 \to End(V),
\end{align*}

\noindent where $\mathfrak{sl}_4$ is the Lie algebra of $SL_4(\C)$. This representation has dimension $45$. If we denote $v_{ij}=z_idz_j-z_jdz_i$, then we can identify the second exterior power of $\C^4$, denoted by $\wedge^2(\C^4)$, with the $\C$-vector space generated by 

$$\Big\{v_{12}, v_{13}, v_{14}, v_{23}, v_{24}, v_{34} \Big\}.$$

\noindent With this, we can see that the representation is a vector subspace of:


$$\wedge^2( \C^4) \otimes Sym^2((\C^4)^*).$$

On the other hand, we must take the generators $L_i$, for $i=1,2,3,4$ 
of the dual vector space $\mathfrak{t}^*$, where $\mathfrak{t}$ is the space of diagonal matrices
of $\mathfrak{sl}_4$. These linear functionals are defined by:

\begin{align*}
L_i: \mathfrak{t} &\to \C \\
(a_1,a_2,a_3,a_4) &\mapsto a_i.
\end{align*}

\noindent We are going to define a norm through the inner product  $\langle L_i, L_j \rangle= \delta_{ij}-\frac{1}{4}$. An element 
$-\alpha=-(b_1L_1+b_2L_2+b_3L_3+b_4L_4) \in \mathfrak{t}^*$ is called weight of the representation if the following vector space has positive dimension:

$$V_{-\alpha}:=\big\{v \in V:  \exp{(a_1,a_2,a_3,a_4)}\cdot v= \exp{(-\alpha(a_1,a_2,a_3,a_4))}v, \forall (a_1,a_2,a_3,a_4) \in  \mathfrak{t} \big\},$$

\noindent in this case, $V_{-\alpha}$ is called the weight space of the representation associated to $-\alpha$. We will say that a non zero vector $v \in V_{-\alpha}$ is a weight vector of  weight $-\alpha$.
\\

We know that the above representation is irreducible  and that $V$ is the direct sum of the weight spaces (see chapter 15 of \cite{fulton}). We list the weight spaces in the table \ref{tablaPesos}, together with their dimension, a basis for each of them, and the length determined by the inner product.

\begin{table}[htbp]
\centering
\scriptsize
\setlength{\tabcolsep}{4pt}
\renewcommand{\arraystretch}{1.1}
\begin{tabular}{|c|c|c|}
\hline
\textbf{Weight} & \textbf{Basis} & $\|\alpha\|^2$ \\
\hline
$-L_3-3L_4$ & $z_4^2v_{34}$ & $6$ \\
$-3L_3-L_4$ & $z_3^2v_{34}$ & $6$ \\
$-L_2-3L_4$ & $z_4^2v_{24}$ & $6$ \\
$-L_2-3L_3$ & $z_3^2v_{23}$ & $6$ \\
$-3L_2-L_4$ & $z_2^2v_{24}$ & $6$ \\
$-3L_2-L_3$ & $z_2^2v_{23}$ & $6$ \\
$-L_1-3L_4$ & $z_4^2v_{14}$ & $6$ \\
$-L_1-3L_3$ & $z_3^2v_{13}$ & $6$ \\
$-L_1-3L_2$ & $z_2^2v_{12}$ & $6$ \\
$-3L_1-L_4$ & $z_1^2v_{14}$ & $6$ \\
$-3L_1-L_3$ & $z_1^2v_{13}$ & $6$ \\
$-3L_1-L_2$ & $z_1^2v_{12}$ & $6$ \\
\hline
$-2L_3-2L_4$ & $z_3z_4v_{34}$ & $4$ \\
$-2L_2-2L_4$ & $z_2z_4v_{24}$ & $4$ \\
$-2L_2-2L_3$ & $z_2z_3v_{23}$ & $4$ \\
$-2L_1-2L_4$ & $z_1z_4v_{14}$ & $4$ \\
$-2L_1-2L_3$ & $z_1z_3v_{13}$ & $4$ \\
$-2L_1-2L_2$ & $z_1z_2v_{12}$ & $4$ \\
\hline
$L_1-L_4$ & $z_4^2v_{23},\, z_3z_4v_{24}$ & $2$ \\
$L_1-L_3$ & $z_3z_4v_{23},\, z_3^2v_{24}$ & $2$ \\
$L_2-L_4$ & $z_4^2v_{13},\, z_3z_4v_{14}$ & $2$ \\
$L_2-L_3$ & $z_3z_4v_{13},\, z_3^2v_{14}$ & $2$ \\
$L_3-L_4$ & $z_4^2v_{12},\, z_2z_4v_{14}$ & $2$ \\
$L_4-L_3$ & $z_3^2v_{12},\, z_2z_3v_{13}$ & $2$ \\
$L_3-L_2$ & $z_2z_4v_{12},\, z_2^2v_{14}$ & $2$ \\
$L_4-L_2$ & $z_2z_3v_{12},\, z_2^2v_{13}$ & $2$ \\
$L_2-L_1$ & $z_1z_4v_{13},\, z_1z_3v_{14}$ & $2$ \\
$L_4-L_1$ & $z_1z_3v_{12},\, z_1z_2v_{13}$ & $2$ \\
$L_3-L_1$ & $z_1z_4v_{12},\, z_1z_2v_{14}$ & $2$ \\
$L_1-L_2$ & $z_2z_4v_{23},\, z_2z_3v_{24}$ & $2$ \\
\hline
$0$ & $z_3z_4v_{12},\, z_2z_4v_{13},\, z_2z_3v_{14}$ & $0$ \\
\hline
\end{tabular}
\caption{Weights of the representation $H^0(\mathbb{P}^3,\Omega^1(4))$.}
\label{tablaPesos}
\end{table}


In order to visualize the weight spaces of this representation in 
$\R^3$, we make the following identification.

\begin{align*}
L_1&\to (-1,1,1)\\
L_2& \to (1,-1,1)\\
L_3& \to (1,1,-1)\\
L_4& \to (-1,-1,-1).
\end{align*}

We then draw all the weights of the representation in $\R^3$. The convex hull of these points provides a useful geometric visualization of the representation (see figure \ref{fig:representation}).

\begin{figure}[htbp]
\centering
\includegraphics[scale=0.1]{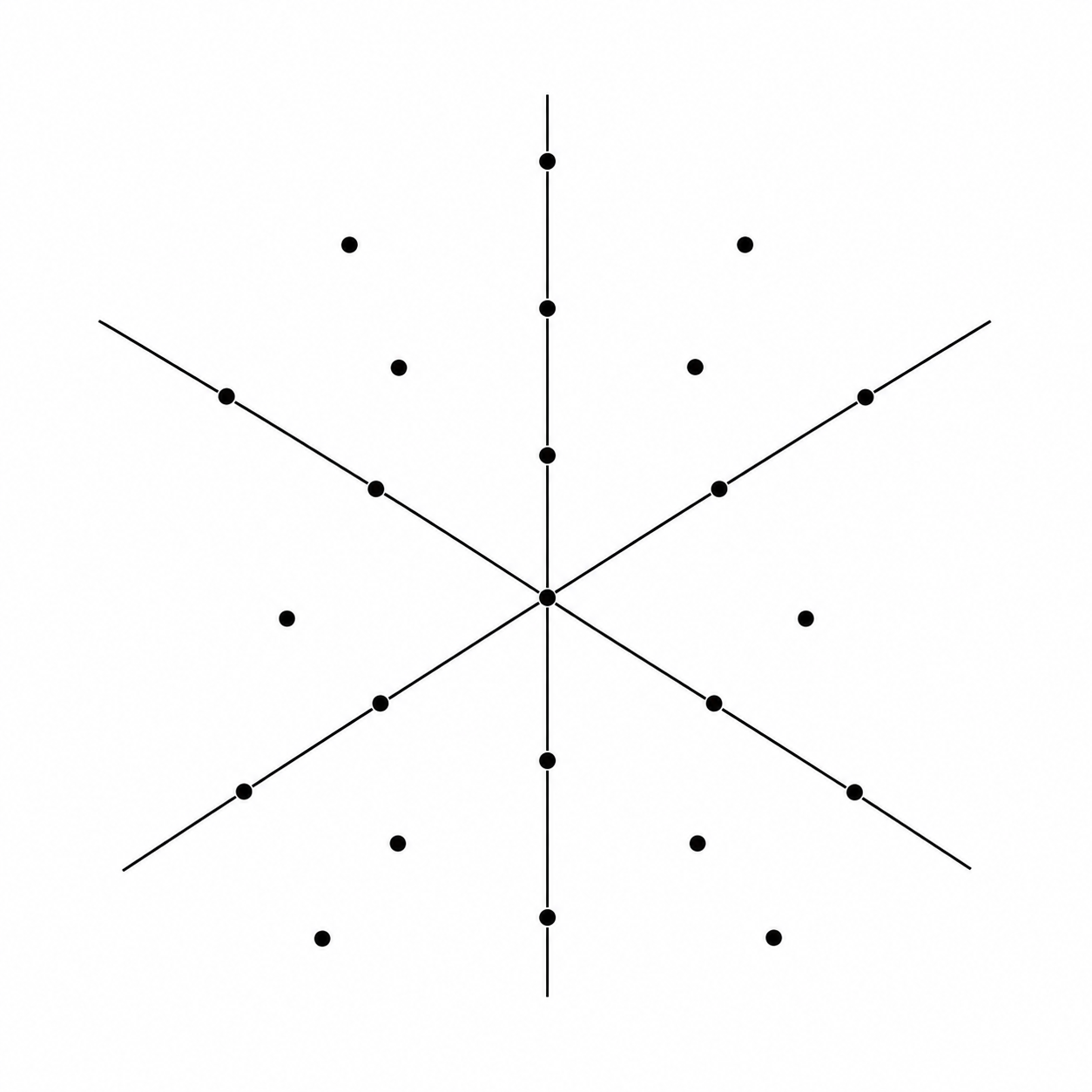}
\caption{\centering The representation $H^0(\mathbb{CP}^3,\Omega_{\mathbb{CP}^3}^1(4))$.}
\label{fig:representation}
\end{figure}

\section{The Automorphism Group of the Exceptional Foliation}



In this section, we study the automorphism group of the exceptional foliation in detail, as a first step toward describing the geometry of the exceptional component.
\\

Recall that $v_{ij}=z_i\,dz_j-z_j\,dz_i$. Consider the exceptional
codimension-one foliation $\mathcal{F}_{\omega}$ of degree $2$ on
$\mathbb{P}^3$, defined by the $1$-form:

$$w=3z_3z_4v_{13}-2z_3^2v_{14}+z_2z_3v_{24}-2z_2z_4v_{23}+z_2z_4v_{14}-3z_4^2v_{12}.$$

\noindent This foliation has singular set $\mathbb{V}(z_4,z_3) \cup \mathbb{V}(z_4,2z_1z_3-z_2^2) \cup \mathbb{V}(2z_2^2-3z_1z_3,3z_1z_4-z_2z_3,z_3^2-2z_2z_4)$ and admits the rational first integral $\frac{(z_1z_4^2-z_2z_3z_4+\frac{z_3^3}{3})^2}{(z_2z_4-\frac{z_3^2}{2})^3}$.
\\

We know that $PGL_4(\C)=Aut(\mathbb{P}^3)$ acts in the space of foliations $\mathbb{F}(2,\mathbb{P}^3)$ by change of coordinates, in fact 

$$\overline{PGL_4(\C)\cdot \omega},$$

\noindent is an irreducible algebraic component of $\mathbb{F}(2,\mathbb{P}^3)$, it is called the exceptional component (see \cite{cl}).
\\

Throughout this article, we work with $SL_4(\mathbb{C})$ instead of $PGL_4(\mathbb{C})$, as this does not affect the results relevant to our study. Since $SL_4(\mathbb{C})$ is isogenous to $PGL_4(\mathbb{C})$, and its one-parameter subgroups are easier to describe explicitly, this choice is technically convenient.
\\

We begin by describing the subgroup of automorphisms of $SL_4(\mathbb{C})$ that preserve the $1$-form $\omega$ up to a nonzero scalar multiple.

\begin{proposition}The automorphisms group of the exceptional foliation $\mathcal{F}_{\omega}$ given by the 1-form $\omega$ is

\begin{align*}
Aut(\mathcal{F}_{\omega})=\Bigg\{
\left(\begin{array}{cccc}
a_{11}&a_{12}&\frac{a_{12}^2a_{33}}{2a_{22}^2}&\frac{a_{12}^3a_{44}}{6a_{22}^3} \\
0&a_{22}&\frac{a_{12}a_{33}}{a_{22}}&\frac{a_{12}^2a_{44}}{2a_{22}^2}\\
0&0&a_{33}&\frac{a_{12}a_{44}}{a_{22}}\\
0&0&0&a_{44}
\end{array}\right) \in SL_4(\C): a_{22}^2=a_{11}a_{33},a_{22}a_{33}=a_{11}a_{44}\Bigg\}.
\end{align*}
\end{proposition}

The proposition follows by a direct computation of the stabilizer of $\mathcal{F}_{\omega}$ under the natural action of $SL_4(\mathbb{C})$. Since the computation is straightforward but lengthy, we omit the details, as they are not essential to the arguments that follow.

We next decompose this automorphism group. This decomposition will be useful throughout the constructions that follow. Let 

\begin{align*}
N=&\Bigg\{n(\alpha)=
\left(\begin{array}{cccc}
1&\alpha&\frac{\alpha^2}{2!}&\frac{\alpha^3}{3!} \\
0&1&\alpha&\frac{\alpha^2}{2!}\\
0&0&1&\alpha\\
0&0&0&1
\end{array}\right): \alpha \in \C \Bigg\}
\end{align*}

\noindent and 

\begin{align*}
D=&\Bigg\{\left(\begin{array}{cccc}
a_{11}&0&0&0 \\
0&a_{22}&0&0\\
0&0&a_{33}&0\\
0&0&0&a_{44}
\end{array}\right) : a_{22}^2=a_{11}a_{33}, a_{22}a_{33}=a_{11}a_{44}, a_{11}a_{22}a_{33}a_{44}=1 \Bigg\}
\end{align*}

\noindent The equations $a_{11}a_{22}a_{33}a_{44}=1$, $a_{22}^2=a_{11}a_{33}$ and $a_{22}a_{33}=a_{11}a_{44}$ leaves just the solution $(a_{11},a_{22},a_{33},a_{44})= (\pm t^3,t,\pm t^{-1},t^{-3})$ for $t \in \C^*$. Then an element in $D$ depends on a parameter and a sign, so we denote it by $d(a_{11}, a_{22})$ to encode all the information. From this, we easily deduce the following lemma.

\begin{lemma}
    Let 
    \begin{align}
g=\left(\begin{array}{cccc}
a_{11}&a_{12}&\frac{a_{12}^2a_{33}}{2a_{22}^2}&\frac{a_{12}^3a_{44}}{6a_{22}^3} \\
0&a_{22}&\frac{a_{12}a_{33}}{a_{22}}&\frac{a_{12}^2a_{44}}{2a_{22}^2}\\
0&0&a_{33}&\frac{a_{12}a_{44}}{a_{22}}\\
0&0&0&a_{44}
\end{array}\right) \in Aut(\mathcal{F}_{\omega}),
\end{align}
    
\noindent then $g=n(a_{12}a_{22}^{-1})d(a_{11},a_{22})$.
\end{lemma}

\begin{proof}
Set $\alpha=a_{12}a_{22}^{-1}$. A direct multiplication gives
\[
n(\alpha)d(a_{11},a_{22})=
\left(
\begin{array}{cccc}
a_{11}&\alpha a_{22}&\frac{\alpha^2a_{33}}{2}&\frac{\alpha^3a_{44}}{6}\\
0&a_{22}&\alpha a_{33}&\frac{\alpha^2a_{44}}{2}\\
0&0&a_{33}&\alpha a_{44}\\
0&0&0&a_{44}
\end{array}
\right).
\]
Substituting $\alpha=a_{12}a_{22}^{-1}$ yields exactly the matrix $g$. Hence $g=n(a_{12}a_{22}^{-1})d(a_{11},a_{22})$.
\end{proof}

This means that we have the following decomposition for $Aut(\mathcal{F}_{\omega})$: 

\begin{align*}
&Aut(\mathcal{F}_{\omega})=\Bigg\{
\left(\begin{array}{cccc}
1&\alpha&\frac{\alpha^2}{2!}&\frac{\alpha^3}{3!} \\
0&1&\alpha&\frac{\alpha^2}{2!}\\
0&0&1&\alpha\\
0&0&0&1
\end{array}\right)\left(\begin{array}{cccc}
a_{11}&0&0&0 \\
0&a_{22}&0&0\\
0&0&a_{33}&0\\
0&0&0&a_{44}
\end{array}\right) :  \\ \quad &\hspace{3cm} \alpha \in \C, a_{22}^2=a_{11}a_{33}, a_{22}a_{33}=a_{11}a_{44} \Bigg\}.
\end{align*}

Now, consider the following action of $D$ on $N$:

\begin{align*}
     D \times N &\to N\\
    \Big(d(a_{11},a_{22}), n(\alpha)\Big) &\mapsto d(a_{11},a_{22})n(\alpha)d(a_{11},a_{22})^{-1}=n(\alpha a_{11}a_{22}^{-1}).
    \end{align*}

With the above discussion, we obtain the following result, which provides an isomorphism between the automorphism group of the exceptional foliation and a semidirect product of two subgroups.

\begin{proposition}\label{decompositionG}
The following map is an isomorphism of algebraic groups between  the semidirect product of $N$ and $D$ and the automorphism group the exceptional foliation:

\begin{align*}
\Phi: N \rtimes D &\to Aut(\mathcal{F}_{\omega})\\
\Big(n(\alpha),d(a_{11},a_{22})\Big) &\mapsto n(\alpha)d(a_{11},a_{22}),
\end{align*}

\noindent where the product in the semidirect product is:

$$\Big(n(\alpha),d(a_{11},a_{22})\Big) \star \Big(n(\beta),d(b_{11},b_{22})\Big)=\Big(n(\alpha)d(a_{11},a_{22})n(\beta)d(a_{11},a_{22})^{-1},d(a_{11},a_{22})d(b_{11},b_{22})\Big).$$
\end{proposition}



To conclude this section, we recall the following previously known result.
Since $\dim Aut(\mathcal{F}_{\omega})=2$, we obtain:

\begin{corollary} The exceptional component $\overline{SL_4(\C) \cdot \omega}$ of $\mathbb{F}(2,\mathbb{P}^3)$ has dimension 13.
\end{corollary}

\section{The Geometry of the Boundary of the Exceptional Component}

Let $G=Aut(\mathcal{F}_{\omega})$ and let $P$ be the subgroup of upper triangular matrices in $SL_4(\C)$, then $G$ is an algebraic subgroup of $P$. We know that, as an algebraic variety, there exists the following isomorphism:
$$SL_4(\C)\cdot \omega \cong SL_4(\C)/G.$$

\noindent and 

$$\overline{SL_4(\C)\cdot \omega} \cong \overline{SL_4(\C)/G}.$$

So, to study the boundary of the exceptional component we are going to study the algebraic variety $\overline{SL_4(\C)/G}$.
\\

Since $G\subset P\subset SL_4(\mathbb C)$, the quotient
$SL_4(\mathbb C)/G$ is the fiber bundle associated to the principal
$P$-bundle $SL_4(\mathbb C)\to SL_4(\mathbb C)/P$ with fiber $P/G$.
Hence
\[
SL_4(\mathbb C)/G \simeq SL_4(\mathbb C)\times_P(P/G),
\]
see, for example, \cite[Section~6]{borel}.
Recall that  $SL_4(\mathbb{C}) \times_P (P/G)$ are the classes in $\Big(SL_4(\mathbb{C}) \times (P/G)\Big)$ considering the action by $P$ given by:

\begin{align*}
   \Big(SL_4(\mathbb{C}) \times (P/G)\Big) \times P &\to SL_4(\mathbb{C}) \times (P/G)\\
    ((h,p_1G),p) &\mapsto (hp,p^{-1}p_1G)
\end{align*}

 Then the isomorphism from $SL_4(\mathbb{C}) \times_P (P/G)$ to $SL_4(\mathbb{C})/G$ sends $[(h,pG)]$ to $hpG$.

 The closure $\overline{SL_4(\mathbb{C})/G}$ is understood in the following sense. 
We consider $SL_4(\mathbb{C})/G$ as a locally closed subvariety of a fixed algebraic variety in which it is naturally embedded. Since
\[
SL_4(\mathbb{C})/G \simeq SL_4(\mathbb{C}) \times_P (P/G),
\]
\noindent we then define
\[
\overline{SL_4(\mathbb{C})/G}
\]
as the Zariski closure of $SL_4(\mathbb{C})/G$ inside the corresponding ambient space, which is compatible with the associated bundle structure. With this we obtain that the Zariski closure of $SL_4(\mathbb{C})/G$ satisfies:

\[
\overline{SL_4(\mathbb{C})/G} \;\cong\; SL_4(\mathbb{C}) \times_P \overline{P/G}.
\]

\noindent This is basically due to the fact that $P/G$ is open and dense in $\overline{P/G}$, and the equality holds locally wherever the bundle is trivial. This means that an element of $\overline{SL_4(\C)/G}$ can be viewed as the class of a pair consisting of an element in $SL_4(\C)$ and a point in the closure $\overline{P/G}$.
\\

We now formalize the above discussion in the following proposition. This will allow us to reduce the study of the boundary of the exceptional component to the study with respect to the action by the subgroup $P$.

\begin{proposition}
Let $P$ be the group of upper triangular matrices in $SL_4(\mathbb{C})$. Then, for every $\nu \in \overline{SL_4 \cdot \omega}$, there exists $g \in SL_4(\C)$ such that
\[
g \cdot \nu \in \overline{P \cdot \omega}.
\]
\end{proposition}

From now on, we focus on studying the algebraic variety $\overline{P/G}$ in order to describe the foliations appearing in the boundary of the exceptional component.
We begin by recalling the well-known Levi decomposition of the group $P$ of upper triangular matrices. Together with the decomposition of the automorphism group $G$ of the exceptional foliation established in the previous section, this description will play a central role in our analysis.
\\

Recall that if

\[
p=
\begin{pmatrix}
a_1 & x_{12} & x_{13} & x_{14} \\
0 & a_2 & x_{23} & x_{24} \\
0 & 0 & a_3 & x_{34} \\
0 & 0 & 0 & a_4
\end{pmatrix}
\in P,
\]

\noindent then

\[
p=
\begin{pmatrix}
1 &
\frac{x_{12}}{a_2} &
\frac{x_{13}}{a_3} &
\frac{x_{14}}{a_4}
\\
0 & 1 &
\frac{x_{23}}{a_3} &
\frac{x_{24}}{a_4}
\\
0 & 0 & 1 &
\frac{x_{34}}{a_4}
\\
0 & 0 & 0 & 1
\end{pmatrix}
\;
\begin{pmatrix}
a_1&0&0&0\\
0&a_2&0&0\\
0&0&a_3&0\\
0&0&0&a_4
\end{pmatrix}.
\]

\noindent With this we obtain the well known Levi decomposition of $P$, which is the semidirect product:
\[
P = U \rtimes T,
\]
where:
\begin{itemize}
\item The unipotent radical $U$ consists of the strictly upper triangular matrices:
    \[
    U = \left\{
    \begin{pmatrix}
    1 & * & * & * \\
    0 & 1 & * & * \\
    0 & 0 & 1 & * \\
    0 & 0 & 0 & 1
    \end{pmatrix}
    \right\}.
    \]  
\item The subgroup $T$ given by the diagonal torus of $SL_4(\mathbb{C})$:
    \[
    T = \left\{
    \begin{pmatrix}
    a_1 & 0 & 0 & 0 \\
    0 & a_2 & 0 & 0 \\
    0 & 0 & a_3 & 0 \\
    0 & 0 & 0 & a_4
    \end{pmatrix}
    \; : \; a_1 a_2 a_3 a_4 = 1
    \right\}.
    \]    
\end{itemize}

\noindent And the action of $T$ on $U$ is by conjugation:

\begin{align*}(a_1,a_2,a_3,a_4)\begin{pmatrix}
    1 & x_{12} & x_{13} & x_{14} \\
    0 & 1 & x_{23} & x_{24} \\
    0 & 0 & 1 & x_{34} \\
    0 & 0 & 0 & 1
    \end{pmatrix}(a_1,a_2,a_3,a_4)^{-1}=\\
    \begin{pmatrix}
    1 & \frac{a_1}{a_2}x_{12} & \frac{a_1}{a_3}x_{13} & \frac{a_1}{a_4}x_{14} \\
    0 & 1 & \frac{a_2}{a_3}x_{23} & \frac{a_2}{a_4}x_{24} \\
    0 & 0 & 1 & \frac{a_3}{a_4}x_{34} \\
    0 & 0 & 0 & 1
    \end{pmatrix},
    \end{align*}

\noindent then for $(u_1,t_1), (u_2,t_2) \in U \rtimes T$, $(u_1,t_1)(u_2,t_2)=(u_1t_1u_2t_1^{-1},t_1t_2)$.
\\

Recall that $G=N \rtimes D$ (see proposition \ref{decompositionG}), therefore, we have the following isomorphism of algebraic varieties:

    \begin{align*}
    P/G=(U \rtimes T)/(N \rtimes D) &\to (U/N) \times (T/D)\\
    (u,t)(N \rtimes D) &\mapsto (uN,tD),
    \end{align*}

\noindent and the action of $P$ on $(U/N) \times (T/D)$ is given by:

\begin{align*}
    (U \rtimes T) \times \Big((U/N) \times (T/D)\Big) &\to (U/N) \times (T/D)\\
    \Big((u,t),(u_1N,t_1D)\Big) &\mapsto (utu_1t^{-1}N,tt_1D).
    \end{align*}

\noindent  We have that $U/N$ is the quotient of $\C^6$ by a linear $\C$, so $U/N$ is isomorphic to $\C^5$. We can construct an explicit isomorphism:

\begin{align*}
&U/N \to \C^5\\
&\begin{pmatrix}
1&x_{12}&x_{13}&x_{14}\\
0&1&x_{23}&x_{24}\\
0&0&1&x_{34}\\
0&0&0&1
\end{pmatrix}N \mapsto \\
&\Big(x_{13}-\frac{x_{12}^2}{2},x_{14}-x_{12}x_{13}+\frac{x_{12}^3}{3},x_{23}-x_{12},x_{24}-x_{12}x_{23}+\frac{x_{12}^3}{2},x_{34}-x_{12}\Big).
\end{align*}

On the other hand, let $C_2$ be the cyclic group $\{1,-1\}$ and $(\C^*)^2/C_2$ be the quotient by the action $C_2 \times (\C^*)^2 \mapsto (\C^*)^2, (\pm 1, (a,b)) \mapsto (\pm a, \pm b)$. We can see that the following:

\begin{align*}
 T/D=\Big\{(a_1,a_2,a_3,a_4)D:(a_1,a_2,a_3,a_4) \in T\Big\} &\to (\C^*)^2/C_2,\\
 (a_1,a_2,a_3,a_4)D &\mapsto [(a_1a_2^{-3},a_2a_3)]
\end{align*}
 
\noindent  is an isomorphism. For that, we can take the element in the class of the diagonal matrix $(a_1,a_2,a_3,a_4)$ such that $a_{22}=a_2^{-1}$, with this we obtain 

\begin{equation*}
(a_1,a_2,a_3,a_4)d(a_{11},a_{22})=(\pm a_1a_2^{-3},1,\pm a_2a_3,a_1^{-1}a_3^{-1}a_2^2).   
\end{equation*}
 
\noindent Therefore, $(a_1a_2^{-3},1, a_2a_3,a_1^{-1}a_3^{-1}a_2^2)$ and $(- a_1a_2^{-3},1,-a_2a_3,a_1^{-1}a_3^{-1}a_2^2)$ are in the same class. 
\\

We now describe the affine variety associated with the quotient $(\mathbb{C}^*)^2/C_2$. For that, we define:

\begin{align*}
    (\C^*)^2/C_2 &\to V=\mathbb{V}_{Aff}(uw-v^2) \subset \mathbb{C}^3\\
    [(a,b)] &\mapsto (a^2,ab,b^2)=(u,v,w).
\end{align*}

\noindent This is an injective algebraic morphism, and the image is 
$V-\big( L_u\cup L_w \big)$, where $L_u=\{(0,0,b): b \in \C\}$ and  $L_w=\{(a,0,0): a \in \C\}$. Therefore, $T/D$ is isomorphic to $V-\big( L_u\cup L_w \big)$.
\\

It is important to mention that neither of these two isomorphisms is canonical, since they depend on a special choice of representatives of the classes.
However, we obtain the following isomorphism for the orbit $P \cdot \omega$:

$$P \cdot \omega \cong P/G \cong \mathbb{A}^5 \times \Big(V-\{(a,0,b): ab=0\}\Big).$$
\\

Hence, the projective compactification of $P/G$ is
\[
\overline{P/G} \;\cong\; \mathbb{P}^5 \times \mathbb{V}(uw-v^2)
\subset \mathbb{P}^5 \times \mathbb{P}^3,
\]
where $\mathbb{V}(uw-v^2)\subset\mathbb{P}^3$ is a quadric cone.
The explicit description of this variety allows us to determine the
irreducible components of the boundary of $\overline{P/G}$, which we
summarize in the following theorem.

\begin{theorem} \label{boundary} Let $\mathcal{F}_{\omega}$ the exceptional foliation, and $P$ be the group of upper triangular matrices, then $\overline{P \cdot \omega}$ is isomorphic to the projective algebraic variety $\mathbb{P}^5 \times \mathbb{V}(uw-v^2)$.
\\

The boundary of $\overline{P \cdot \omega}$ has $4$ irreducible components, which are isomorphic to the following algebraic varieties:

 \begin{align*}
 Z_1&=\mathbb{P}^5 \times \overline{L_u}\\
 Z_2&=\mathbb{P}^5 \times C_{\infty}\\
     Z_3 &=\mathbb{P}^5 \times \overline{L_w} \\
     Z_4&=H_{\infty} \times \mathbb{V}(uw-v^2)
 \end{align*}

\noindent where $C_{\infty}=\{(u:v:w:0) \in \mathbb{P}^3: uw=v^2\}$, $\overline{L_u}=\{(0:0:w:z) \in \mathbb{P}^3\},$ $\overline{L_w}=\{(u:0:0:z) \in \mathbb{P}^3\}$ and $H_{\infty}$ is the hyperplane to infinity of $\mathbb{P}^5$, that is, the complement of the standard affine chart $\mathbb{A}^5$. 

\end{theorem}

The above description of $\overline{P \cdot \omega}$ also allows us to determine its singular set, as stated in the following corollary.

\begin{corollary}
    The singular set of $\overline{P \cdot \omega}$ is isomorphic to $\mathbb{P}^5 \times \{(0:0:0:1)\}$. In particular, the dimension of the singular set of the exceptional component $\overline{E}$ is $11$.
\end{corollary}

\begin{proof}
By compatibility of the singular locus with the fiber bundle we have:
\[
\operatorname{Sing}\bigl(SL_4(\mathbb{C})\times_P \overline{P/G}\bigr)
=
SL_4(\mathbb{C})\times_P \operatorname{Sing}(\overline{P/G}).
\]

Since $\operatorname{Sing}(\overline{P/G})
\cong \mathbb{P}^5 \times \{(0:0:0:1)\}$, we obtain
\[
\dim \operatorname{Sing}\bigl(\overline{SL_4(\mathbb{C})/G}\bigr)
=
\dim(SL_4(\mathbb{C})/P) + 5=11.
\]
\end{proof}

\begin{remark} It is easy to see that every irreducible component of $\mathbb{F}(2, \mathbb{P}^3)$ is unirational and ruled. For example, to see that the exceptional component is ruled we remark that the rational first integral $F=\frac{(z_1z_4^2-z_2z_3z_4+\frac{z_3^3}{3})^2}{(z_2z_4-\frac{z_3^2}{2})^3}$ is, in some sense, linear in $z_1$. If we take the automorphisms of $\mathbb{P}^3$ given by:

$$\sigma_{t}:(z_1:z_2:z_3:z_4) \mapsto (t^3z_1:t^{-1}z_2:t^{-1}z_3:t^{-1}z_4) \qquad t \in \C^*$$

\noindent then $(F\circ \sigma_t) (z_1:z_2:z_3:z_4)=\frac{(t^4z_1z_4^2-z_2z_3z_4+\frac{z_3^3}{3})^2}{(z_2z_4-\frac{z_3^2}{2})^3}$. By considering $t \mapsto \sigma_t^{*}\mathcal{F}_{\omega}$, we construct a line in the exceptional component. Using the action of $SL_4(\C)$, we see that this component is ruled.
\end{remark}

We recall the following conjecture of Brunella, which concerns the
existence of algebraic invariant subvarieties for codimension-one
foliations on $\mathbb{P}^3$.

\begin{conjecture}
Let $\mathcal{F}$ be a codimension one holomorphic foliation on $\mathbb{P}^{3}$. Then either 

\begin{enumerate}
\item[a)] $\mathcal{F}$ admits an invariant algebraic surface, 
\item[b)] or there exists a one-dimensional foliation by algebraic curves contained in $\mathcal{F}$.
\end{enumerate}
\end{conjecture}

\noindent and according to Cerveau's result (see Proposition~1 of \cite{Cerveau2002}):

\begin{proposition}\label{Brunella}
Let $\mathcal{F}$ be a generic degree $s$ codimension one holomorphic foliation on $\mathbb{P}^{3}$ belonging to a linear pencil. Then we have one of the following alternative:

\begin{enumerate}
\item[(a')] There exists a rational closed $1$-form $\theta$ such that $d\omega=\theta\wedge\omega$ (the foliation is transversally affine).
\item[(b')] $\mathcal{F}$ satisfies (b).
\end{enumerate}
\end{proposition}

Therefore, every foliation in $\mathbb{F}(2,\mathbb{P}^3)$ satisfies Proposition~\ref{Brunella} and Brunella's conjecture. In fact, in \cite{cerveau-loray-pereira-touzet}, the authors exhibit an example of a transversely projective foliation that does not satisfy Proposition~\ref{Brunella}. One open question is the following: does there exist a foliation $\mathcal{F}\in\mathbb{F}(d,\mathbb{P}^3)$ with finite or trivial automorphism group such that
the closure of its orbit $\overline{SL_4(\mathbb{C})\cdot\mathcal{F}}$ is an irreducible component of $\mathbb{F}(d,\mathbb{P}^3)$?

\section{Foliations in the boundary of the Exceptional Component}

We now begin the analysis of orbits in the boundary of the exceptional component, which, as we have proved, are contained in $\overline{P\cdot \omega}$, up to a change of coordinates. 
\\

Our exceptional foliation is defined by the 1-form $\omega$, and this is the sum of:

\begin{align*}
&\omega_1=z_2\big(z_3v_{24}-2z_4v_{23}\big) \\
&\omega_2=z_3\big(3z_4v_{13}-2z_3v_{14}\big) \\
&\omega_3=z_4\big(z_2v_{14}-3z_4v_{12}\big).
\end{align*}

\noindent Observe that the vectors generating 
$\omega_i$ have both weight $L_i-L_{i+1}$, for $i=1,2,3$, each of the corresponding weight spaces has dimension $2$. These three weights lie in the plane defined by $y+2z-2=0$ in $\R^3$ (see figure \ref{fig:representation}). 
\\

We can directly compute the weights that appear in a generic element of $\overline{P \cdot \omega}$, which will lie in the half-space
$\{y+2z-2 \geq 0\}$. Below we present a table (see \ref{Table2}) listing all of them, together with their corresponding points in $\mathbb{R}^3$ and the associated vectors. We also analyze the convex hull they generate, together with the faces of the resulting polytope and the elements contained in each face.
\\

Recall that an element in $\overline{P \cdot \omega}$ will be a linear combination of the vectors in the list of table \ref{Table2}, the coefficients must satisfy certain algebraic conditions, which determine membership in the component.

\begin{table}[h]
\centering
\begin{tabular}{c|c|c}
\hline
weight & $\mathbb{R}^3$ & vector\\
\hline
$-L_3-3L_4$ & $(2,2,4)$ & $z_4^2v_{34}$\\
$-3L_3-L_4$ & $(-2,-2,4)$ & $z_3^2v_{34}$\\
$-L_2-3L_4$ & $(2,4,2)$ & $z_4^2v_{24}$\\
$-L_1-3L_4$ & $(4,2,2)$ & $z_4^2v_{14}$\\
$-2L_3-2L_4$ & $(0,0,4)$ & $z_3z_4v_{34}$\\
$-2L_2-2L_4$ & $(0,4,0)$ & $z_2z_4v_{24}$\\
$L_1-L_4$ & $(0,2,2)$ & $z_4^2v_{23},\, z_3z_4v_{24}$\\
$L_1-L_3$ & $(-2,0,2)$ & $z_3z_4v_{23},\, z_3^2v_{24}$\\
$L_2-L_4$ & $(2,0,2)$ & $z_4^2v_{13},\, z_3z_4v_{14}$\\
$L_2-L_3$ & $(0,-2,2)$ & $z_3z_4v_{13},\, z_3^2v_{14}$\\
$L_3-L_4$ & $(2,2,0)$ & $z_4^2v_{12},\, z_2z_4v_{14}$\\
$L_1-L_2$ & $(-2,2,0)$ & $z_2z_4v_{23},\, z_2z_3v_{24}$\\
\hline
\end{tabular}
\caption{Weights of $P \cdot \omega$}
\label{Table2}
\end{table}

\noindent Therefore, for the construction of the polytope we need the following set of points in $\mathbb{R}^3$:

\[
\begin{aligned}
S=\Big\{&
(2,2,4),\,(-2,-2,4),\,(2,4,2),\,(4,2,2),\\
&(0,0,4),\,(0,4,0),\,(0,2,2),\,(-2,0,2),\\
&(2,0,2),\,(0,-2,2),\,(2,2,0),\,(-2,2,0)
\Big\}.
\end{aligned}
\]

The polytope corresponding to the variety $\overline{P \cdot \omega}$
is $C=\operatorname{Conv}(S) \subset \R^3$.
Once these points are determined, the geometry of the convex polytope
$C$ can be described explicitly. 

 Its vertices are the points

\[
\operatorname{Vert}(C)=
\{
(2,2,4),\,(-2,-2,4),\,(2,4,2),\,(4,2,2),\,
(0,4,0),\,(0,-2,2),\,(2,2,0),\,(-2,2,0)
\}.
\]

 Moreover, $C$ admits the following description:

\[
C=
\left\{
(x,y,z)\in\mathbb{R}^3:
\begin{array}{rcl}
0 &\leq& x+y+z \leq 8,\\[2mm]
-x+y+z &\leq& 4,\\
x-y+z &\leq& 4,\\
x+y-z &\leq& 4,\\
x-y-z &\leq& 0,\\
z &\geq& 0,\\
y+2z &\geq& 2
\end{array}
\right\}.
\]

Hence, the faces of \(C\) are supported by the following planes, which contain the following points of $S$:

\begin{align*}
y+2z=2 \quad &:\quad \{(0,-2,2),\,(2,2,0),\,(-2,2,0)\},\\[1mm]
x+y+z=0 \quad &:\quad \{(-2,-2,4),\,(-2,0,2),\,(0,-2,2),\,(-2,2,0)\},\\[1mm]
z=0 \quad &:\quad \{(0,4,0),\,(2,2,0),\,(-2,2,0)\},\\[1mm]
x-y-z=0 \quad &:\quad \{(4,2,2),\,(2,0,2),\,(0,-2,2),\,(2,2,0)\},\\[1mm]
-x+y+z=4 \quad &:\quad \{(2,2,4),\,(-2,-2,4),\,(2,4,2),\,(0,0,4),\\
&\qquad\qquad\{(0,4,0),\,(0,2,2),\,(-2,0,2),\,(-2,2,0)\},\\[1mm]
x-y+z=4 \quad &:\quad \{(2,2,4),\,(-2,-2,4),\,(4,2,2),\\
&\qquad\qquad (0,0,4),\,(2,0,2),\,(0,-2,2)\},\\[1mm]
x+y-z=4 \quad &:\quad \{(2,4,2),\,(4,2,2),\,(0,4,0),\,(2,2,0)\},\\[1mm]
x+y+z=8 \quad &:\quad \{(2,2,4),\,(2,4,2),\,(4,2,2)\}.
\end{align*}

The geometry of the orbit closure $\overline{P\cdot\omega}$ is reflected in the combinatorics of the associated weight polytope. In particular, weight vectors whose weights lie on faces of the polytope correspond to degenerations of foliations appearing in the boundary of $\overline{P\cdot\omega}$. These degenerations arise as limits under the action of one-parameter subgroups of $SL_4(\mathbb{C})$. Recall that a one-parameter subgroup of $SL_4(\mathbb{C})$ is a homomorphism of algebraic groups
\[
\lambda:\mathbb{C}^*\longrightarrow SL_4(\mathbb{C}).
\]

Thus, the combinatorial structure of the polytope encodes the possible degenerations of the exceptional component.
More precisely, the following proposition holds.

\begin{proposition} Every element $\nu \in \overline{P \cdot \omega}$ is, up to a change of coordinates, a linear combination of weight vectors supported on a face of the polytope $C$.
\end{proposition}

\begin{proof} The proof of this proposition relies heavily on the framework developed by Kirwan in \textit{Cohomology of Quotients in Symplectic and Algebraic Geometry}~\cite{kirwan}. We omit the technical details, as they are not central to the purposes of this article.
\\

In our case, the exceptional foliation $\mathcal{F}_{\omega}$ corresponds to an unstable point (in the sense of GIT) for the $P$ action by change of coordinates, due to the fact that its automorphism group has dimension $2$. Consider the diagonal one-parameter subgroup of $SL_4(\C)$, defined by $\lambda_0(t)=(t^{3},t,t^{-1},t^{-3})$, it is easy to see that $\lambda_0(t) \omega=t^{-2} \omega$. Then $\lambda_0$ is the one-parameter subgroup that fixes the foliation and determines the stratum in which it lies. The image of $\lambda_0$ is contained in the automorphism group of the foliation. 
\\

According to Kirwan’s notation (see Theorem 12.26, definition 12.18 and lemma 12.16 of \cite{kirwan}), this implies that $P\cdot\omega = Y_{\lambda_0}$, and this is a locally closed subset.
In our case, this set is contained in the subspace generated by the weight vectors lying in the half-space $y+2z-2 \ge 0$ associated with $\lambda_0$, namely the vectors listed in table \ref{Table2}. To have  $P\cdot\omega$ we must to impose certain open conditions on the coefficients of vectors in table \ref{Table2}.
\\
Taking the closure of $P\cdot\omega = Y_{\lambda_0}$, as in Kirwan’s construction, amounts to removing these openness conditions and allowing all coefficients to vanish (see the proof of Lemma 12.16 in \cite{kirwan}). Geometrically, this corresponds to taking limits along the various faces of the weight polytope.
\end{proof}

The weight vectors whose weights lie on the faces of the polytope, define the vector spaces in which the foliations in the boundary lie, up to change of coordinates. Associated with each face, there exists a one-parameter subgroup corresponding to the point of minimal norm on that face, with respect to the inner product given by $\langle L_i, L_j \rangle= \delta_{ij}-\frac{1}{4}$.
With the isomorphisms given by this inner product, we can see that $L_i$ corresponds to the following diagonal one-parameter subgroups:

\begin{align*}
L_1 \to (t^3,t^{-1},t^{-1},t^{-1})\\
L_2 \to (t^{-1},t^{3},t^{-1},t^{-1})\\
L_3 \to (t^{-1},t^{-1},t^{3},t^{-1})\\
L_4 \to (t^{-1},t^{-1},t^{-1},t^3)\\
\end{align*}

In our case, the polytope has eight faces. However, we disregard the last three $x-y+z=4$, $x+y-z=4$ and $x+y+z=8$, since they define foliations whose singular sets have dimension at least $2$, whereas in our definition we require foliations to have singular sets of dimension at most $1$.

The following table lists the relevant faces together with their associated one-parameter subgroups.

\begin{table}[ht]
\centering
\begin{tabular}{c c}
\hline
Face in $\R^3$ & Associated $1$-PS \\
\hline

0. $y+2z=2$
&
$\lambda_0(t)=(t^{3},t,t^{-1},t^{-3})$
\\[2mm]

1. $x+y+z=0$
&
$\lambda_1(t)=(t,1,t^{-1},1)$
\\[2mm]

2. $x-y-z=0$
&
$\lambda_2(t)=(1,t,1,t^{-1})$
\\[2mm]

3. $z=0$
&
$\lambda_3(t)=(t,t^{-1},t,t^{-1})$
\\[2mm]

4. $-x+y+z=4$
&
$\lambda_4(t)=(t^{3},t^{-1},t^{-1},t^{-1})$
\\
\hline
\end{tabular}
\caption{One-parameter subgroups associated with the relevant faces of the weight polytope of $\overline{P \cdot \omega}$}
\label{tab:kirwan-1ps}
\end{table}

For $i=0,\dots,4$, let $W_i$ denote the vector subspace of $H^0(\mathbb{P}^3,\Omega^1(4))$ generated by the weight vectors lying on the $i$-th face of the polytope. Then the following limits exist and define, up to a change of coordinates, the foliations contained in the exceptional component:

\begin{align*}
P\cdot\omega &\longrightarrow W_i,\\
p\cdot\omega &\longmapsto \lim_{t\to \infty}\lambda_i(t)(p\cdot\omega).
\end{align*}

Moreover, this limit defines a foliation that is invariant under the one-parameter subgroup $\lambda_i(t)$.
\\

The 0-face contains the weight vectors generating the entire foliation $\omega$, namely $\omega_1$, $\omega_2$, and $\omega_3$. Therefore, the corresponding limit is a foliation in the orbit of the exceptional foliation, yielding $P\cdot\omega$. In fact, as mentioned above, we have that $\lambda_0(t)\in \operatorname{Aut}(\omega)$ for every $t\in\mathbb{C}^*$. The conclusion is that the face corresponing to $\lambda_0$ is the open set $P \cdot \omega$ in $\overline{P\cdot \omega}$.

\subsection{Faces 1,2,3}

The faces corresponding to $\lambda_i$, for $i=1,2,3$ contain only two weights among $L_i - L_{i+1}$. Remember that $\omega=\omega_1+\omega_2+\omega_3$, where:
 
 \begin{align*}
&\omega_1=z_2z_3v_{24}-2z_2z_4v_{23} \\
&\omega_2=3z_3z_4v_{13}-2z_3^2v_{14} \\
&\omega_3=z_4z_2v_{14}-3z_4^2v_{12},
\end{align*}

\noindent and consider the diagonal 1-parameter subgroup:

\begin{align*}
    \lambda:\C^* &\to P \subset SL_4(\C)\\
    t &\mapsto (t^{n_1},t^{n_2},t^{n_3},t^{n_4}).
\end{align*}

As we said before: $\lim_{t \to 0} \lambda(t) \cdot \omega_i=t^{n_i-n_{i+1}}\omega_i,$ for $i=1,2,3$. Then

\begin{align*}
\lambda_1(t) \cdot \omega&=t\omega_1+t\omega_2+t^{-1}\omega_3\\
\lambda_2(t) \cdot \omega&=t^{-1}\omega_1+t \omega_2+t \omega_3\\
\lambda_3(t) \cdot \omega&=t^2 \omega_1+t^{-2}\omega_2+t^{2}\omega_3.
\end{align*}

\noindent Therefore

\begin{align*}
\lim_{t \to \infty} \lambda_1(t) \cdot \omega&=\omega_1+\omega_2\\
\lim_{t \to \infty} \lambda_2(t) \cdot \omega&=\omega_2+\omega_3\\
\lim_{t \to \infty} \lambda_3(t) \cdot \omega&=\omega_1+\omega_3.
\end{align*}

We conclude that the foliations $\omega_2+\omega_3$, $\omega_1+\omega_3$, $\omega_1+\omega_2$ are in $\overline{P \cdot \omega}-(P \cdot \omega)$. It can be computed directly that the automorphism groups of these foliations are as follows:

\begin{align*}
&Aut(\omega_1+\omega_2)=\\
&  \left\{\left(\begin{array}{cccc}
a_{11}&a_{12}&a_{13}&0 \\
0&a_{22}&a_{23}&0\\
0&0&a_{33}&0\\
0&0&0&a_{44}
\end{array}\right) : 3a_{22}a_{23}-a_{12}a_{33}=0, 3a_{23}^2-2a_{13}a_{33}=0, a_{22}^2-a_{11}a_{33}=0 \right\}
\end{align*}

\begin{align*}
&Aut(\omega_2+\omega_3)=\\
&  \left\{\left(\begin{array}{cccc}
a_{11}&a_{12}&0&0 \\
0&a_{22}&0&0\\
0&0&a_{33}&a_{34}\\
0&0&0&a_{44}
\end{array}\right) : a_{12}a_{44}-3a_{22}a_{34}=0, a_{11}a_{44}-a_{22}a_{33}=0 \right\}
\end{align*}

\begin{align*}
&Aut(\omega_1+\omega_3)=\\
&  \left\{\left(\begin{array}{cccc}
a_{11}&0&0&0 \\
0&a_{22}&a_{23}&a_{24}\\
0&0&a_{33}&a_{34}\\
0&0&0&a_{44}
\end{array}\right) : a_{22}a_{44}-a_{33}^2=0, a_{23}a_{44}-2a_{44}a_{33}=0, a_{24}a_{44}-a_{34}^2=0 \right\}
\end{align*}

\noindent In each case, we have that the dimension of the automorphism group is 3 and, as expected, it is a subgroup of the parabolic group $P$. Then its orbit has dimension $6$ in $\overline{P\cdot \omega}$.

\begin{theorem} We have the following isomorphisms of locally closed varieties

\begin{align*}
P \cdot (\omega_1+\omega_2) &\cong \mathbb{A}^5 \times \overline{L}_u-\{(0:0:1:0),(0:0:0:1)\} \subset Z_1\\
P \cdot (\omega_2+\omega_3) &\cong \mathbb{A}^5 \times \overline{L}_w-\{(0:0:0:1),(1:0:0:0)\} \subset Z_3\\
P \cdot (\omega_1+\omega_3) &\cong \mathbb{A}^5 \times C_{\infty}-\{(0:0:1:0),(1:0:0:0)\} \subset Z_2
\end{align*}
\end{theorem}

\begin{proof} We can directly see that the limits associated with the diagonal one-parameter subgroups $\lambda_i$, for $i=1,2,3$ correspond, in the variety $\mathbb{V}(uw-v^2)$ (which is isomorphic to the closure of the quotient $T/D$), to:

\begin{align*}
    \lambda_1 &\to \overline{L}_u\\
    \lambda_2 &\to \overline{L}_w\\
    \lambda_3 &\to C_{\infty}.    
\end{align*}

Therefore, these are degenerations of the diagonal part of the group P, it follows that the orbits $P \cdot (\omega_2+\omega_3)$, $P \cdot(\omega_1+\omega_3)$, $P \cdot(\omega_1+\omega_2)$ correspond to dense open subsets of components $Z_1, Z_2, Z_3$ (see Theorem \ref{boundary}) beacuse all have dimension $6$. In fact they are subsets in $\mathbb{A}^5 \times \overline{L}_u$, $\mathbb{A}^5 \times \overline{L}_w$ and $\mathbb{A}^5 \times C_{\infty}$, the affine part $\mathbb{A}^5$ corresponds to letting $P$ act on $\omega_i+\omega_j$, for $i \neq j$.
\end{proof}

We can say more about these foliations: they are logarithmic of type $(1,1,2)$, and we can see that they have the following singular set and rational first integral.

\begin{equation*}
\renewcommand{\arraystretch}{1.1}
\small
\begin{array}{c|c|c|c}
\textrm{Foliation} & \textrm{Singular Set} & \textrm{Rational First Integral} & \textrm{Logarithmic form} \\
\hline
\omega_1+\omega_2 &
\mathbb{V}(z_4,z_3)\cup \mathbb{V}(z_3,z_2)\cup \mathbb{V}(z_4,2z_1z_3-z_2^2) &
\frac{(2z_1z_3-z_2^2)z_4^2}{z_3^4} &
-4\frac{dz_3}{z_3}
+2\frac{dz_4}{z_4}
+\frac{d(2z_1z_3-z_2^2)}{2z_1z_3-z_2^2}
\\
\hline
\omega_2+\omega_3 &
\mathbb{V}(z_4,z_1)\cup \mathbb{V}(z_4,z_3)\cup \mathbb{V}(z_1,2z_2z_4-z_3^2) &
\frac{(2z_2z_4-z_3^2)^3}{z_1^2z_4^4} &
-2\frac{dz_1}{z_1}
-4\frac{dz_4}{z_4}
+3\frac{d(2z_2z_4-z_3^2)}{2z_2z_4-z_3^2}
\\
\hline
\omega_1+\omega_3 &
\mathbb{V}(z_4,z_2)\cup \mathbb{V}(z_4,z_3)\cup \mathbb{V}(z_2,z_1) &
\frac{(z_1z_4-z_2z_3)^2}{z_2^3z_4} &
-3\frac{dz_2}{z_2}
-\frac{dz_4}{z_4}
+2\frac{d(z_1z_4-z_2z_3)}{z_1z_4-z_2z_3}
\end{array}
\end{equation*}

The foliations $\omega_1$, $\omega_2$ and $\omega_3$ have a singular set of dimension greater than 2, and therefore they are not properly regarded as points in the space of foliations $\mathbb{F}(2,\mathbb{P}^3)$.

\subsection{Face 4}

It remains to reach points in the last irreducible component of the boundary, namely $Z_4=H_{\infty} \times \mathbb{V}(uw-v^2)$ (see Theorem \ref{boundary}). For this purpose, we must study degenerations arising from the 1-parameter subgroup $\lambda_4(t)=(t^{3},t^{-1},t^{-1},t^{-1})$. For that we use the following general lemma:

\begin{lemma}
Consider the 1-Parameter subgroup:

\begin{align*}
    \lambda_4:\C^* &\to P \subset SL_4(\C), \quad t \mapsto (t^{3},t^{-1},t^{-1},t^{-1}).
\end{align*}
\noindent A foliation $\nu \in \mathbb{F}(d,\mathbb{P}^3)$ does not depend on the variable 
$z_1$ if and only if: $$\lambda_4(t) \cdot \nu=t^{d+2} \omega.$$

\noindent Moreover, for any foliation $\nu \in \mathbb{F}(d,\mathbb{P}^3)$ the foliation given by 

$$\lim_{t \mapsto \infty} \lambda_4(t) \cdot \nu,$$

\noindent does not depend on the variable 
$z_1$.
\end{lemma}

\begin{proof}
Suppose that $\nu=\sum_{i=1}^4 A_i(z_1,z_2,z_3,z_4)dz_i$ defines a foliation in $\mathbb{F}(d,\mathbb{P}^3)$. It is clear that if $\nu$ does not depend on $z_1$, then the action is $\lambda_1(t) \cdot \nu=t^{d+2} \nu$. Conversely, if $\nu$ contains a term involving $z_1$ or $dz_1$, then that term has weight
\[
(d+2-i)-3i = d+2-4i,
\]
for some $i\in\mathbb{N}$. Since $d+2-4i < d+2$, then 
$\lim_{t \mapsto \infty} \lambda_1(t) \cdot \nu$ consists only of terms that are independent of $z_1$.
\end{proof}

With this, we conclude that all degenerations corresponding to face 4 are foliations which, up to a change of coordinates, do not depend on the variable $z_1$, and hence they are pull-back foliations.
\\

A foliation of this type contains in its automorphism group matrices of the form

$\begin{pmatrix}
a^3 & * & * & *\\
0 & a^{-1} & 0 & 0 \\
0 & 0 & a^{-1} & 0\\
0 & 0 & 0 & a^{-1}
\end{pmatrix}, $

\noindent where $a \in \C^*$. This means that its automorphism group has dimension 4 and therefore its orbit in $P$ has dimension 5. Hence, the last component $Z_4$ (see Theorem \ref{boundary}) is formed by a 1-dimensional family of pull-back foliations. All of the above can be summarized in the following theorem.

\begin{theorem}
The irreducible component $H_{\infty} \times \mathbb{V}(uw-v^2)$ of the boundary of the exceptional component contains a dense open subset consisting of infinitely many linear pull-back foliations from $\mathbb{P}^2$.
\end{theorem}

We now give the explicit construction of this 1-dimensional family of pull-back foliations lying in the boundary of the exceptional component.
\\

For the construction we take the rational first integral 
$$\frac{(z_1z_4^2-z_2z_3z_4+\frac{z_3^3}{3})^2}{(z_2z_4-\frac{z_3^2}{2})^3},$$

\noindent of the exceptional foliation given by $\omega$. We work in affine chart $z_4=1$, then we have, $\frac{(z_1-z_2z_3+\frac{z_3^3}{3})^2}{(z_2-\frac{z_3^2}{2})^3}$ and we consider the change of coordinates: $z_1 \to \epsilon z_1+az_2+bz_3, z_2 \to z_2, z_3 \to z_3$. If we make this change and let $\epsilon \to 0$, we obtain:

$$F_{a,b} = \frac{(a z_2 + b z_3 - z_2 z_3 + \frac{z_3^3}{3})^2}{(z_2 - \frac{z_3^2}{2})^3}.$$

\noindent Which depends only on $z_2$ and $z_3$, this means that it defines a rational first integral of a foliation pull-back from 
$\mathbb{P}^2$. By construction: for every $(a,b) \in \C^2$, $F_{a,b}$ is a rational first integral of a foliation belonging to the boundary of the exceptional component.

\begin{proposition}
Let $(a,b) \in \C^2$ and consider the pencil in $\mathbb{P}^2$ given by:

$$F_{a,b} =\frac{P_{a,b}^2}{Q^3},$$

\noindent where $P_{a,b} = a z_2z_4^2 + b z_3z_4^2 - z_2 z_3z_4 + \frac{z_3^3}{3},$
and $Q = z_2z_4 - \frac{z_3^2}{2}$. For generic  $(a,b), (a',b') \in \C^2$ the pencils 
corresponding to $F_{a,b}$ and $F_{a',b'}$ are equivalent if and only if
$(\frac{a}{a'})^2 = \frac{b}{b'}$
\end{proposition}

\begin{proof} The pencils $F_{a,b}$ and $F_{a',b'}$ are equivalent if there exists a linear map
$\Phi: \mathbb{P}^2 \to \mathbb{P}^2$ and a projective transformation $R \in \mathrm{PGL}(2)$ such that
\[
F_{a',b'} \circ \Phi = R(F_{a,b}).
\]

\noindent Since the unique non reduced fibers of the pencil are $Q^3$ and $P_{a,b}^2$, then $R$ must be a constant. We see that $\Phi$ must preserve the line $z_4$ because this is invariant by the foliation, this means that, locally, we can view the transformation $\Phi$ as:

\[
\Phi(z_2,z_3) = (\alpha z_2 + \beta z_3 + \gamma,\; \delta z_3 + \nu z_2 + \lambda).
\]

It must also preserve the conic $Q$, since it does not depend on the parameters $a,b$. The absence of $z_2^2$ terms in $Q$ forces $\nu=0$ and we have also 
$\alpha=\delta^2$, $\beta=\delta\lambda$ and $2\gamma=\lambda^2$. Then the linear morphisms preserving Q and $z_4$ has the following form:
\[
\Phi(z_2,z_3)=(\delta^2 z_2+ \delta \lambda z_3 + \frac{\lambda^2}{2},\delta z_3 + \lambda),\]

\noindent and we obtain:

\begin{align*}
&P_{a,b}(\delta^2 z_2+ \delta \lambda z_3 + \frac{\lambda^2}{2},\delta z_3 + \lambda)=\\
&\delta^3\Bigg(\frac{a-\lambda}{\delta}z_2+\frac{\lambda a+b-\frac{\lambda^2}{2}}{\delta^2}z_3-z_2z_3+\frac{z_3^3}{3}+\lambda \Big(\frac{\lambda a}{2}+b -\frac{\lambda^2}{6} \Big)\Bigg)=\delta^3 P_{a',b'}(z_2,z_3).
\end{align*}

\noindent Since $(a,b)\in\mathbb{C}^2$ is generic, we conclude that $\lambda=0$; and we have $a'=\frac{a}{\delta}, b'=\frac{b}{\delta^2}$.  
Therefore the pencils given by $F_{a,b}$ and $F_{a',b'}$ are equivalent if and only if $(\frac{a}{a'})^2=\frac{b}{b'}$.
\end{proof}

This means that, generically, distinct points in $\C^2$ determine non-conjugate foliations of degree $2$, which are pull-back foliations from $\mathbb{P}^2$ in the boundary of the exceptional component.  Then we have:

\begin{corollary} For $(a,b) \in \C^2$, the rational function

$$F_{a,b} = \frac{(a z_2z_4^2 + b z_3z_4^2 - z_2 z_3z_4 + \frac{z_3^3}{3})^2}{(z_2z_4 - \frac{z_3^2}{2})^3},$$

\noindent is the rational first integral of a pull-back type foliation in the boundary of the exceptional component $\overline{E}$ of $\mathbb{F}(2,\mathbb{P}^3)$. The foliations corresponding to $(a,b)$ and $(a',b') \in \C^2$ are in the same orbit if and only if $a^2b'=ba'^2$. 
    \end{corollary}

\backmatter

\bmhead{Acknowledgements}
The first author gratefully acknowledges the Université de Rennes for its hospitality during the preparation of this work.


\bigskip

\begin{appendices}




\end{appendices}


\bibliography{sn-bibliography}

\end{document}